\documentclass[reqno,10pt]{amsart}

\usepackage{epsfig}
\usepackage{amssymb}
\usepackage{xcolor}
\usepackage{enumitem}
\newtheorem{thm}{Theorem}%[section]

\newtheorem{lem}[thm]{Lemma}
\newtheorem{cor}[thm]{Corollary}

\theoremstyle{definition}
\newtheorem{rem}[thm]{Remark}
\newtheorem*{note}{Note}

\newcommand{\R}{{\mathbb{R}}}

\begin{document}

\title[Gr\"unbaum's inequality]{Gr\"unbaum's inequality for sections}

\author{S.~Myroshnychenko, M.~Stephen, and N.~Zhang}

\address{Sergii Myroshnychenko, Department of Mathematical and Statistical Sciences, University of Alberta, Edmonton, Alberta, T6G 2G1, Canada}
\email{myroshny@ualberta.ca}
\address{Matthew Stephen, Department of Mathematical and Statistical Sciences, University of Alberta, Edmonton, Alberta, T6G 2G1, Canada}
\email{mastephe@ualberta.ca}
\address{Ning Zhang, Mathematical Sciences Research Institute, Berkeley, CA, 94720-5070 United States}
\email{nzhang2@ualberta.ca}



%\thanks{The second and third named authors were supported by an NSERC grant, and by NSF Grant No. DMS-1440140}

\begin{abstract}
We show
\begin{align*}
\frac{ \int_{E \cap \theta^+} f(x) dx }{ \int_E f(x) dx } 
	\geq \left(\frac{k \gamma+1}{(n+1) \gamma+1}\right)^{\frac{k \gamma+1}{\gamma}} 
\end{align*}
for all $k$-dimensional subspaces $E\subset\R^n$, $\theta\in E\cap S^{n-1}$, and all $\gamma$-concave functions $f:\R^n\rightarrow [0,\infty)$ with $\gamma >0$, $0< \int_{\R^n} f(x)\, dx <\infty$, and $\int_{\R^n} x f(x)\, dx$ at the origin $o\in\R^n$. Here, $\theta^+ := \lbrace x\, : \, \langle x,\theta\rangle \geq 0 \rbrace$. As a consequence of this result, we get the following generalization of Gr\"unbaum's inequality:
\begin{align*}
\frac{ \mbox{vol}_k(K\cap E\cap\theta^+) }{ \mbox{vol}_k(K\cap E) } \geq \left( \frac{k}{n+1} \right)^k
\end{align*}
for all convex bodies $K\subset\R^n$ with centroid at the origin, $k$-dimensional subspaces $E\subset\R^n$, and $\theta\in E\cap S^{n-1}$. The lower bounds in both of our inequalities are the best possible, and we discuss the equality conditions.
\end{abstract}

\maketitle

\section{Introduction}

An elegant inequality of Gr\"unbaum \cite{G1} gives a lower bound for the volume of that portion of a convex body lying in a halfspace which slices the convex body through its centroid. Let $K$ be a {\it convex body} in $\R^n$; that is, a convex and compact set with non-empty interior. Assume that the {\it centroid} of $K$,
\begin{align*}
g(K) := \frac{1}{\mbox{vol}_n(K) } \int_K x\, dx \in\mbox{int}(K) ,
\end{align*}
is at the origin. Given a unit vector $\theta\in S^{n-1}$, we define $\theta^+ := \lbrace x\, : \, \langle x,\theta\rangle \geq 0 \rbrace$. Specifically, Gr\"unbaum's inequality states that
\begin{align}\label{Grunbaum Ineq}
\frac{ \mbox{vol}_n(K\cap\theta^+) }{ \mbox{vol}_n(K) } \geq \left( \frac{n}{n+1} \right)^n .
\end{align}
There is equality when, for example, $K$ is the cone
\begin{align*}
\mbox{conv} \left( \frac{-1}{n+1} \theta + B_2^{n-1}, \, \frac{n}{n+1} \theta \right)
\end{align*}
and $B_2^{n-1}$ is the unit ball in $\theta^\perp$. This volume inequality was independently proven in \cite{M}.

Compare Gr\"unbaum's inequality with the following long known lower bound for the distance between the centroid $g(K) = o$ and a supporting hyperplane of $K$. The {\it support function} of $K$ is defined by $h_K(x) = \max_{y\in K} \langle x,y\rangle$ for $x\in\R^n$. Evaluated at the unit vector $\theta$, $h_K(\theta)$ gives the distance from the origin to the supporting hyperplane of $K$ in the direction $\theta$. Now, the aforementioned inequality is
\begin{align}\label{MinkRad ineq}
\frac{ h_K(\theta) }{ h_K(-\theta) + h_K(\theta) } \geq \frac{1}{n+1} .
\end{align}
There is equality when, for example, $K$ is the cone
\begin{align*}
\mbox{conv} \left( \frac{-n}{n+1} \theta + B_2^{n-1}, \, \frac{1}{n+1} \theta \right) .
\end{align*}
Refer to pages 57--58 of \cite{BF} for a discussion of (\ref{MinkRad ineq}).

A generalization of Gr\"unbaum's inequality was recently established in \cite{SZ} for projections of a convex body. Let $E$ be a $k$-dimensional subspace of $\R^n$, $1\leq k\leq n$, and let $K|E$ denote the orthogonal projection of $K$ onto $E$. ``Gr\"unbaum's inequality for projections" states
\begin{align}\label{Grunbaum Ineq Proj}
\frac{ \mbox{vol}_k \Big( (K|E)\cap\theta^+ \Big) }{ \mbox{vol}_k (K|E) } \geq \left( \frac{k}{n+1}\right)^k .
\end{align}
There is equality when, for example,
\begin{align}\label{Grunbaum Equality}
K = \mbox{conv}\left( \left( 1 - \frac{k}{n+1}\right)\theta + B_2^{k-1} ,\ \frac{k}{n+1}\theta + B_2^{n-k} \right) ,
\end{align}
$\theta\in E\cap S^{n-1}$, $B_2^{k-1}$ is the unit ball in $E\cap\theta^\perp$, and $B_2^{n-k}$ is the unit ball in $E^\perp$. See Corollary 9 in \cite{SZ} for the complete characterization of the equality conditions. Observe that Gr\"unbaum's inequality for projections provides a link between inequalities (\ref{Grunbaum Ineq}) and (\ref{MinkRad ineq}). Let us also emphasize that Gr\"unbaum's inequality does not imply (\ref{Grunbaum Ineq Proj}) because the centroid of $K|E$ is in general different from the centroid of $K$.

One goal of our paper is to establish a ``Gr\"unbaum's inequality for sections" with equality conditions (see Figure \ref{general}). What is the largest constant $c = c(n,k)>0$, depending only on $n$ and $k$, so that
\begin{align}\label{Grunbaum Ineq Sec}
\frac{ \mbox{vol}_k(K\cap E\cap\theta^+) }{ \mbox{vol}_k(K\cap E) } \geq c ?
\end{align}
\begin{figure}[ht]
  \centering
  \includegraphics[scale=0.2]{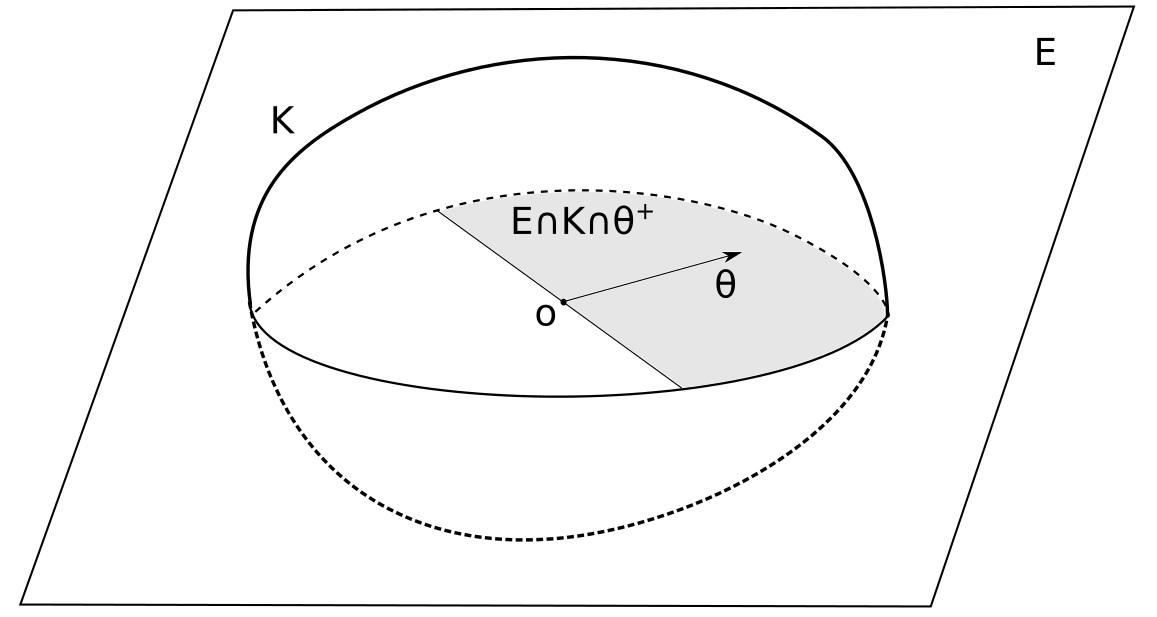}\\
  \caption{How small is $E \cap K \cap \theta^+$ compared to $E \cap K$?}\label{general}
\end{figure}

This question was first asked by Fradelizi, Meyer, and Yaskin in \cite{FMY}. They showed there is an absolute constant $c_0 >0$ so that
\begin{align*}
c \geq c_1 := \frac{ c_0 }{ (n-k+1)^2 } \left( \frac{k}{n+1}\right)^{k-2} ,
\end{align*}
but they did not prove $c = c_1$. Again, note that the value of $c$ cannot be obtained from Gr\"unbaum's inequality because the centroid of $K\cap E$ is in general different from the centroid of $K$. Given Gr\"unbaum's inequality for projections, it was conjectured in \cite{SZ} that $c = \left(\frac{k}{n+1}\right)^k$. To our knowledge, it is not possible to verify this conjecture directly from inequality (\ref{Grunbaum Ineq Proj}) for $1<k<n$.

We prove in this paper that $c = \left(\frac{k}{n+1}\right)^k$; see Corollary \ref{Grunbaum Ineq Sec Thm}. There is equality in (\ref{Grunbaum Ineq Sec}) when, for example, $K$ has the form in (\ref{Grunbaum Equality}). The complete characterization of the equality conditions is given in Corollary \ref{Grunbaum Ineq Sec Thm}. Gr\"unbaum's inequality for sections links inequalities (\ref{Grunbaum Ineq}) and (\ref{MinkRad ineq}), once it is observed that (\ref{MinkRad ineq}) is equivalent to
\begin{align*}
\frac{ \rho_K(\theta) }{ \rho_K(-\theta) + \rho_K(\theta) } \geq \frac{1}{n+1} .
\end{align*}

There are also several functional versions of Gr\"unbaum's inequality. If the function $f:\R^n\rightarrow [0,\infty)$ is integrable, {\it log-concave} (i.e. $\log f$ is concave on convex support), and such that  $\int_{\R^n} xf(x)\, dx = o$, then
\begin{align*}
\int_{\theta^+} f(x)\, dx \geq \frac{1}{e} \int_{\R^n} f(x)\, dx .
\end{align*}
This inequality is the limiting case of Gr\"unbaum's inequality when the dimension tends to infinity, where $f$ is approximated with an appropriate sequence of convex bodies $K_m\subset\R^m$. Refer to Lemma 2.2.6 in \cite{BGVV} for an alternative proof.

Our main result extends another functional version of Gr\"unbaum's inequality proven in \cite{MNRY}. It was shown in \cite{MNRY} that
\begin{align*}
\int_0^\infty f(s\theta) \, ds \geq \frac{1}{e^n} \int_{-\infty}^\infty f(s\theta) \, ds
\end{align*}
for every log-concave $f:\R^n\rightarrow [0,\infty)$ with a finite and positive integral, and $\int_{\R^n}xf(x) \, dx = o$. We say a function $f:\R^n\rightarrow [0,\infty)$ is $\gamma$-concave for $\gamma>0$ if $f^\gamma$ is concave on convex support. Adapting the methods used in \cite{MNRY}, we find $\tilde{c} = \tilde{c}(n,\gamma) = \left( \frac{ \gamma + 1 }{ \gamma n +\gamma + 1 } \right)^\frac{ \gamma +1 }{\gamma} >0 $ is the best constant so that
\begin{align*}
\int_0^\infty f(s\theta) \, ds \geq \tilde{c} \int_{-\infty}^\infty f(s\theta) \, ds
\end{align*}
for every $\gamma$-concave function $f:\R^n\rightarrow [0,\infty)$, $\gamma >0$, with $0<\int_{\R^n} f(x)\, dx <\infty$ and $\int_{\R^n} xf(x)\, dx = o$. See Theorem \ref{main} for the precise statement and the characterization of the equality conditions.

We state and prove Theorem \ref{main}, an integral inequality for one dimensional sections of $\gamma$-concave functions, in Section 2. In Section 3, we show that Theorem \ref{main} implies a corresponding integral inequality for $k$-dimensional sections of $\gamma$-concave functions: 
\begin{align*}
\int_{E \cap \theta^+} f(x) dx \geq \left(\frac{k \gamma+1}{(n+1) \gamma+1}\right)^{\frac{k \gamma+1}{\gamma}} \int_E f(x) dx
\end{align*}
for every $k$-dimensional subspace $E$ in $\R^n$, $\theta\in E\cap S^{n-1}$, and every $\gamma$-concave $f:\R^n\rightarrow [0,\infty)$ with $\gamma >0$, $0<\int_{\R^n} f(x)\, dx <\infty$, and $\int_{\R^n} xf(x)\, dx = o$. In Section 4, we prove Gr\"unbaum's inequality for sections as another consequence of Theorem \ref{main}.

\section{ One Dimensional Sections of $\gamma$-Concave Functions }

A function $f:\mathbb{R}^n\rightarrow [0,\infty)$ is {\it $\gamma$-concave} for $\gamma\in (-\infty,0)\cup(0,\infty)$ if
\begin{align}\label{gamma def}
f\big( \lambda x + (1-\lambda)y \big) \geq \big[ \lambda f^\gamma(x) + (1-\lambda) f^\gamma(y) \big]^\frac{1}{\gamma}
\end{align}
for all $0\leq \lambda\leq 1$ and all $x,y\in\mathbb{R}^n$ such that $f(x) \cdot f(y)\neq 0$. We say $f$ is {\it $\gamma$-affine} if inequality (\ref{gamma def}) is always an equality. These definitions are extended to $\gamma = 0,\, \pm\infty$ by continuity, and log-concavity corresponds to the case $\gamma = 0$. The support of a function $f$ will be denoted by $K_f := \mbox{supp}(f)$. If $f$ is $\gamma$-concave, then $K_f$ is a convex set. If $f$ is $\gamma$-concave for some $\gamma\in (0,\infty)$ with a positive and finite integral, then $K_f$ is a convex body in $\mathbb{R}^n$ (see Remark 2.2.7 (i) in \cite{BGVV}); in this case, we define the {\it centroid} of $f$ by
\begin{align*}
g(f):=\int_{\R^n} x f(x)\, dx \bigg/ \int_{\R^n} f(x)\,dx \in\mbox{int}(K_f) .
\end{align*}

\begin{note}
We will always implicitly assume that a $\gamma$-concave function is continuous on its support. This does not lead to a real loss of generality in our results. Indeed, a $\gamma$-concave $f$ must be continuous on the (relative) interior of $K_f$; assuming $f$ is continuous on $K_f$ at most requires a redefinition of $f$ on a set of measure zero.
\end{note}

Our main result is the following theorem:

\begin{thm}\label{main}
Fix $\theta\in S^{n-1}$ and $\gamma\in (0,\infty)$. Let $f:\mathbb{R}^n\rightarrow [0,\infty)$ be a $\gamma$-concave function with $0<\int_{\R^n} f(x)\, dx<\infty$ and centroid at the origin. Then
\begin{align*}
\frac{ \int_0^\infty f(s\theta) \, ds }{ \int_{-\infty}^\infty f(s\theta) \, ds }
	\geq \left( \frac{ \gamma + 1 }{ \gamma n +\gamma + 1 } \right)^\frac{ \gamma +1 }{\gamma} .
\end{align*}
There is equality if and only if
\begin{itemize}
\item $f(x) = m \mathcal{X}_{K_f}(x) \Big( - \langle x,\xi\rangle + r \langle \theta,\xi\rangle \Big)^\frac{1}{\gamma}$ for some constants $m,r >0$ and a unit vector $\xi\in S^{n-1}$ such that $\langle\theta,\xi\rangle >0$;
\item $K_f = \mbox{conv} \left( - \left( \frac{n\gamma}{\gamma +1}\right) r\theta, \ r \theta + D \right)$ for some $(n-1)$-dimensional convex body $D\subset\xi^\perp$ whose centroid (taken in $\xi^\perp$) is at the origin.
\end{itemize}
\end{thm}

For the remainder of Section 2 we fix $\theta\in S^{n-1}$, $\gamma\in (0,\infty)$, and a $\gamma$-concave $f:\mathbb{R}^n\rightarrow [0,\infty)$ satisfying the hypotheses of Theorem \ref{main}. We prove Theorem \ref{main} in subsections 2.1 to 2.3 by transforming $f$ into a function having the form from the equality case, while showing that the ratio $\int_{\langle g(f),\theta\rangle}^\infty f(s\theta) \, ds / \int_{-\infty}^\infty f(s\theta) \, ds$ can only decrease.

\subsection{Replacing $\gamma$-concave slices with $\gamma$-affine slices}

For each $x'\in K_f|\theta^\perp$, define $f_{x'}:\R\rightarrow [0,\infty)$ to be the one dimensional restriction $f_{x'}(s):=f(x'+s\theta)$. We will transform each slice $f_{x'}$ into a $\gamma$-affine function of the form
\begin{align}\label{affine_slice}
F_{x'}(s):=\mathcal{X}_{\left[\Psi(x'),\frac{H(x')}{\beta}\right]}(s)(-\beta s+H(x'))^\frac{1}{\gamma} ,
\end{align}
where $\Psi,H: K_f|\theta^\perp\rightarrow\R$ are functions and $\beta>0$ is a constant. As the first step in constructing $F_{x'}$, choose
\begin{align*}
\beta := \frac{ \gamma f_o^{\gamma + 1}(0) }{ (\gamma + 1) \int_0^\infty f_o(s)\, ds } > 0
\end{align*}
so that
\begin{align}\label{beta}
\int_0^\infty f_o(s)\, ds = \left( \frac{\gamma f_o^{\gamma +1}(0)}{\gamma +1} \right) \frac{1}{\beta}
	= \int_0^{\frac{f^\gamma_o(0)}{\beta}} (-\beta s+f^\gamma_o(0))^\frac{1}{\gamma}\, ds .
\end{align}

Before describing $H$, we introduce the auxiliary function $\widetilde{H}:K_f|\theta^\perp\rightarrow\R$ defined by
\begin{align*}
\widetilde{H}(x') := \max_{ a\in\mbox{supp}(f_{x'})} \widetilde{H}(x';a) \qquad \mbox{for} \qquad x'\in K_f|\theta^\perp ,
\end{align*}
where
\begin{align*}
\widetilde{H}(x';a) := \left[ \frac{\beta(\gamma+1)}{\gamma}
	\int_{a}^{\infty} f_{x'}(s) \, ds \right]^\frac{\gamma}{\gamma+1} + \beta a .
\end{align*}
The function $\widetilde{H}$ is well-defined with $\widetilde{H}(x')\in\mathbb{R}$ for every $x'\in K_f|\theta^\perp$, because $\mbox{supp}(f_{x'})$ is a compact interval and $\widetilde{H}(x';\cdot):\R\rightarrow\R$ is continuous. For the moment, fix $a\in \mbox{supp} (f_{x'})$ and $h\geq\widetilde{H}(x')$. It follows from the definition of $\widetilde{H}$ that $h\geq \beta a$. Furthermore, 
\begin{align}\label{h>H}
\int_{a}^{\infty} f_{x'}(s) \,ds < \frac{\gamma}{ \beta ( \gamma + 1 )} (-\beta a + h)^\frac{\gamma +1}{\gamma}
	= \int_{a}^{\frac{h}{\beta}} (-\beta s+h)^\frac{1}{\gamma} \, ds 
\end{align}
if and only if $h > \widetilde{H}(x')$ or $h = \widetilde{H}(x') > \widetilde{H}(x';a)$, and
\begin{align}\label{h=H}
\int_{a}^{\infty} f_{x'}(s) \,ds = \frac{\gamma}{ \beta ( \gamma + 1 )} (-\beta a + h)^\frac{\gamma +1}{\gamma}
	= \int_{a}^\frac{h}{\beta} (-\beta s+ h )^\frac{1}{\gamma} \, ds
\end{align}
if and only if $h = \widetilde{H}(x') = \widetilde{H}(x';a)$. More generally,
\begin{align}\label{h>=H}
\int_{a}^{\infty} f_{x'}(s) \,ds \leq \int_{a}^\infty \chi_{\big( -\infty,\frac{h}{\beta}\big]}(s)
	(-\beta s+h)^\frac{1}{\gamma} \, ds 	
\end{align}
for all $a\in\R$ and $h\geq \widetilde{H}(x')$.

We now prove that $\widetilde{H}(o) = f_o^\gamma(0)$. The function $f_o^\gamma:\R\rightarrow [0,\infty)$ is concave on its support,
\begin{align*}
l(s):= \mathcal{X}_{\big(-\infty, \frac{f^\gamma_o(0)}{\beta}\big]} (s) (-\beta s+f^\gamma_o(0))
\end{align*}
is affine on its support, and $f_o^\gamma(0) = l(0)$; these facts and equality (\ref{beta}) imply there is some $0< s' < f^\gamma_o(0) / \beta$ for which
\begin{align*}
&f_o(s) < (-\beta s+f^\gamma_o(0))^\frac{1}{\gamma} \quad \mbox{whenever} \quad s<0 , \\
&f_o(s) > (-\beta s+f^\gamma_o(0))^\frac{1}{\gamma} \quad \mbox{whenever} \quad 0<s<s' , \\
\mbox{and} \quad &f_o(s) < (-\beta s+f^\gamma_o(0))^\frac{1}{\gamma} \quad \mbox{whenever} \quad
	s' < s < \frac{f^\gamma_o(0)}{\beta} .
\end{align*}
It follows that $\mbox{supp}(f_o)\subset \big( -\infty, f_o^\gamma(0)/\beta \big]$ and
\begin{align*}
\int_a^\infty \left( l^\frac{1}{\gamma}(s) - f_o(s) \right) ds \geq 0 \quad \mbox{for all} \quad a\in\R.
\end{align*}
Therefore, if $\widetilde{H}(o) > f_o^\gamma(0)$, then
\begin{align*}
\int_a^\frac{\widetilde{H}(o)}{\beta} (-\beta s+ \widetilde{H}(o) )^\frac{1}{\gamma} \, ds
	> \int_a^{\frac{f^\gamma_o(0)}{\beta}} (-\beta s+f^\gamma_o(0))^\frac{1}{\gamma}\, ds
	\geq \int_a^\infty f_o(s)\, ds
\end{align*}
for every $a\in\mbox{supp}(f_o)$. Choosing $a\in\mbox{supp}(f_o)$ so that $\widetilde{H}(o) = \widetilde{H}(o,a)$, this last inequality contradicts (\ref{h=H}). On the other hand, if $\widetilde{H}(o)< f_o^\gamma(0)$, then equation (\ref{beta}) contradicts (\ref{h>H}).

We claim $\widetilde{H}$ is concave on $K_f|\theta^\perp$. Indeed, let $0\leq\lambda\leq 1$ and $x'_1,x'_2\in K_f|\theta^\perp$. For $j = 1,2$, choose $a_j\in\mbox{supp}(f_{x_j'})$ so that $\widetilde{H}(x_j') = \widetilde{H}(x_j';a_j)$. The Borell-Brascamp-Lieb inequality (see, for example, Theorem 10.1 in \cite{Ga1}), equality (\ref{h=H}), and inequality (\ref{h>=H}) then imply
\begin{align*}
&   \left[\frac{\gamma}{\beta (\gamma+1)}\right]^{\frac{\gamma}{\gamma+1}}
\left(-\beta \big(\lambda a_1 +(1-\lambda)a_2\big) +\widetilde{H}\big(\lambda x'_1 + (1-\lambda)x'_2\big)\right) \nonumber \\
=&\left[\int_{\lambda a_1+(1-\lambda)a_2}^{\frac{\widetilde{H}\big(\lambda x'_1 + (1-\lambda)x'_2\big)}{\beta}}(-\beta s + \widetilde{H}\big(\lambda x'_1 + (1-\lambda)x'_2\big))^\frac{1}{\gamma}\,ds\right]
^{\frac{\gamma}{\gamma+1}} \nonumber \\
\geq & \left[\int_{\lambda a_1+(1-\lambda)a_2}^{\infty}f(\lambda x'_1+(1-\lambda)x'_2+s\theta)\,ds\right]
	^{\frac{\gamma}{\gamma+1}} \nonumber \\
\geq & \lambda \left[\int_{a_1}^{\infty}f(x_1'+s\theta)\,ds\right]^{\frac{\gamma}{\gamma+1}}+(1-\lambda) 		
	\left[\int_{a_2}^{\infty}f(x_2'+s\theta)\,ds\right]^{\frac{\gamma}{\gamma+1}} \nonumber \\
= & \lambda \left[\int_{a_1}^{\frac{\widetilde{H}(x_1')}{\beta}}(-\beta s + \widetilde{H}(x_1'))^\frac{1}{\gamma}\,ds
	\right]^{\frac{\gamma}{\gamma+1}} +(1-\lambda) \left[\int_{a_2}^{\frac{\widetilde{H}(x_2')}{\beta}}
	(-\beta s + \widetilde{H}(x_2'))^\frac{1}{\gamma}\,ds\right]^{\frac{\gamma}{\gamma+1}} \nonumber \\
= &\lambda\left[\frac{\gamma}{\beta(\gamma+1)}\right]^{\frac{\gamma}{\gamma+1}}(-\beta a_1 + \widetilde{H}(x_1'))
	+ (1-\lambda) \left[\frac{\gamma}{\beta (\gamma+1)}\right]^{\frac{\gamma}{\gamma+1}}
	(-\beta a_2 + \widetilde{H}(x_2')) \nonumber \\
= & \left[\frac{\gamma}{\beta (\gamma+1)}\right]^{\frac{\gamma}{\gamma+1}}
	\left(-\beta \big(\lambda a_1 +(1-\lambda)a_2\big) +\lambda \widetilde{H}(x_1')
	+ (1-\lambda) \widetilde{H}(x_2')\right).
\end{align*}
Therefore, we must have
\begin{align*}
\widetilde{H}\big(\lambda x'_1 + (1-\lambda)x'_2\big)\geq \lambda \widetilde{H}(x_1') + (1-\lambda) \widetilde{H}(x_2').
\end{align*}

As $\widetilde{H}:K_f|\theta^\perp\rightarrow\R$ is concave, there is a linear function $L:\theta^\perp\to \R$ such that $\widetilde{H}(x')\leq \widetilde{H}(o) + L(x')$ for every $x'\in K_f|\theta^\perp$. Recalling that $\widetilde{H}(o) = f^\gamma(o)$, we now put
\begin{align*}
H(x'):=f^\gamma(o)+L(x') \quad \mbox{for all} \quad x'\in K_f|\theta^\perp ,
\end{align*}
so that $H$ is an affine function on $K_f|\theta^\perp$.

Having defined $\beta>0$ and $H:K_f|\theta^\perp\rightarrow\R$, we finally choose $\Psi(x')\in\R$ so that $\Psi(x') \leq H(x')/\beta$ and
\begin{align}\label{EqualMass}
\int_{-\infty}^\infty F_{x'}(s)\, ds = \int_{-\infty}^\infty f_{x'}(s)\, ds,
\end{align}
where $F_{x'}(s)$ is defined as in (\ref{affine_slice}). Then,
\begin{align*}
\frac{\gamma}{ \beta (\gamma +1) } (-\beta\Psi(x')+H(x'))^\frac{\gamma+1}{\gamma} = \int_{\Psi(x')}^{\frac{H(x')}{\beta}}
	(-\beta s+H(x'))^\frac{1}{\gamma} \, ds=\int_{-\infty}^\infty f(x'+s\theta)\,ds,
\end{align*}
which gives
\begin{align*}
\Psi(x')=\frac{1}{\beta}\left[f^\gamma(o) + L(x') - \left(\frac{\beta(\gamma+1)}{\gamma}
	\int_{-\infty}^\infty f(x'+s\theta)\,ds\right)^{\frac{\gamma}{\gamma+1}}\right].
\end{align*}
Since $L(x')$ is linear and $x'\mapsto \int_\mathbb{R} f(x'+s\theta)\,ds$ is $\frac{\gamma}{\gamma+1}$-concave (again by the Borell-Brascamp-Lieb inequality), we have that $\Psi:K_f|\theta^\perp\rightarrow\R$ is convex.

Now, define the function $F:\R^n\rightarrow [0,\infty)$ by
\begin{align}\label{F}
F(x) := \mathcal{X}_{K_F}(x) \Big( -\beta \langle x,\theta\rangle + H\big(x-\langle x,\theta\rangle\theta\big)\Big)^\frac{1}{\gamma}
	\quad \mbox{for} \quad x\in\R^n .
\end{align}
Here,
\begin{align*}
K_F = \bigg\lbrace x\in\R^n : \, x' = \big( x - \langle x,\theta\rangle \theta \big) \in K_f|\theta^\perp
	\mbox{ and } \Psi(x')\leq \langle x,\theta\rangle \leq \frac{ H(x') }{\beta} \bigg\rbrace ,
\end{align*}
and $\beta$, $H$, $\Psi$ are as previously constructed. The set $K_F$ is a convex body in $\R^n$ with $K_F|\theta^\perp = K_f|\theta^\perp$, because $\Psi,H:K_f|\theta^\perp\rightarrow\R$ are, respectively, convex and concave with $\Psi<H/\beta$ on the relative interior of $K_f|\theta^\perp$, and $\Psi\leq H/\beta$ on the relative boundary. Therefore, it is clear that $F$ is $\gamma$-affine with $\mbox{supp}(F) = K_F$.

Also note that $F(x'+s\theta) \equiv F_{x'}(s)$ for each $x'\in K_F|\theta^\perp$, where $F_{x'}$ is the $\gamma$-affine slice defined in (\ref{affine_slice}). Equality (\ref{beta}) remains true if the right-hand side is replaced with $\int_0^\infty F_o(s)\, ds$ because of $H(o) = f^\gamma(o)$ and the choice of $\Psi(o)$ in (\ref{EqualMass}). Similarly, (\ref{h>=H}) is still valid when the right-hand side is replaced with $\int_{a}^\infty F_{x'}(s)\, ds$. When we reference (\ref{beta}) and (\ref{h>=H}) in the proof of the next lemma, we will be referring to their altered forms.

\begin{lem}\label{affine}
The centroid $g(F)$ lies on the $\theta$-axis and $\langle g(F),\theta\rangle \geq 0$. Furthermore,
\begin{align}\label{F_ratio}
\frac{ \int_{\langle g(F),\theta\rangle}^\infty F(s\theta)\, ds}{\int_{-\infty}^\infty F(s\theta)\, ds }
	\leq \frac{ \int_0^\infty F(s\theta)\, ds}{\int_{-\infty}^\infty F(s\theta)\, ds }
	= \frac{ \int_0^\infty f(s\theta)\, ds}{ \int_{-\infty}^\infty f(s\theta)\, ds} ,
\end{align}
with equality if and only if $K_F = K_f$ and $F\equiv f$.
\end{lem}

\begin{proof}
Because the mass of $F$ along lines parallel to $\R \theta$ is the same as for $f$ (see equation (\ref{EqualMass})), $g(F)$ will lie on the $\theta$-axis. Integration by parts, the fact that $H(x')\geq \widetilde{H}(x')$, and inequality (\ref{h>=H}) together imply
\begin{align}\label{F_centroid1}
\int_{-\infty}^\infty s\big( F_{x'}(s) - f_{x'}(s) \big) \, ds
	= \int_{-\infty}^\infty \int_t^\infty \big( F_{x'}(s) - f_{x'}(s) \big) \, ds \, dt \geq 0
\end{align}
for all $x'\in K_f|\theta^\perp$. Inequality (\ref{F_centroid1}), equation (\ref{EqualMass}), and $K_F|\theta^\perp = K_f|\theta^\perp$ now show
\begin{align}\label{F_centroid2}
\langle g(F),\theta\rangle = \frac{ \int_{\R^n} \langle x,\theta\rangle F(x)\, dx }{ \int_{\R^n} F(x)\, dx }
	&= \frac{ \int_{K_F|\theta^\perp} \int_{-\infty}^\infty s F_{x'}(s)\, ds \, dx' }
	{ \int_{K_F|\theta^\perp} \int_{-\infty}^\infty F_{x'}(s)\, ds \, dx' } 	\nonumber \\
&\geq \frac{ \int_{K_f|\theta^\perp} \int_{-\infty}^\infty s f_{x'}(s)\, ds \, dx' }
	{ \int_{K_f|\theta^\perp} \int_{-\infty}^\infty f_{x'}(s)\, ds \, dx' } = \langle g(f),\theta\rangle = 0 .
\end{align}
Inequality (\ref{F_centroid2}), equation (\ref{beta}), and equation (\ref{EqualMass}) immediately give (\ref{F_ratio}).

Suppose there is equality in (\ref{F_ratio}). Equality in (\ref{F_ratio}) is only possible if $\langle g(F),\theta\rangle = 0$, which implies equality in (\ref{F_centroid2}). It then follows from inequality (\ref{F_centroid1}) and the equality in (\ref{F_centroid2}) that
\begin{align*}
&\int_{K_f|\theta^\perp} \left| \int_{-\infty}^\infty s \big( F_{x'}(s) - f_{x'}(s) \big) \, ds \right| \, dx' \\
&= \int_{K_F|\theta^\perp} \int_{-\infty}^\infty s F_{x'}(s) \, ds \, dx'
	- \int_{K_f|\theta^\perp} \int_{-\infty}^\infty s f_{x'}(s) \, ds \, dx' = 0 .
\end{align*}
With continuity, we necessarily have
\begin{align*}
\int_{-\infty}^\infty s \big( F_{x'}(s) - f_{x'}(s) \big) \, ds = 0
\end{align*}
for every $x'\in K_f|\theta^\perp$, so there is equality in (\ref{F_centroid1}). Inequality (\ref{h>=H}) and the equality in (\ref{F_centroid1}) imply
\begin{align*}
\int_{-\infty}^\infty \left| \int_t^\infty \big( F_{x'}(s) - f_{x'}(s) \big) \, ds \right| \, dt
	= \int_{-\infty}^\infty \int_t^\infty \big( F_{x'}(s) - f_{x'}(s) \big) \, ds \, dt = 0 .
\end{align*}
Again invoking continuity, we get that
\begin{align*}
\int_t^\infty f_{x'}(s) \, ds = \int_t^\infty F_{x'}(s) \, ds
\end{align*}
for all $x'\in K_f|\theta^\perp$ and $t\in\R$, so the supports of $F$ and $f$  must coincide. We conclude $F\equiv f$, after differentiating both sides of the last equation with respect to $t$.
\end{proof}

\subsection{Replacing the domain with a cone}

Let $q:\mathbb{R}^n\rightarrow [0,\infty)$ be a $\gamma$-affine function with centroid at the origin, and having the form
\begin{align*}
q(x) = \mathcal{X}_{K_q}(x)
	\Big( -\alpha \langle x,\theta\rangle + G\big(x-\langle x,\theta\rangle\theta\big)\Big)^\frac{1}{\gamma} ;
\end{align*}
$\alpha >0$ is any positive constant, and $K_q = \mbox{supp}(q)$ is any convex body such that
\begin{align*}
K_q = \bigg\lbrace x\in\R^n : \, x' = \big( x - \langle x,\theta\rangle \theta \big) \in K_q|\theta^\perp
	\mbox{ and } \Phi(x')\leq \langle x,\theta\rangle \leq \frac{ G(x') }{\alpha} \bigg\rbrace
\end{align*}
for some respectively convex and affine functions $\Phi,\,G:K_q|\theta^\perp\rightarrow\R$. Distinct level sets of $q$ lie within distinct but parallel hyperplanes, because $q$ is $\gamma$-affine. Also, the set $\lbrace q(x) = 0\rbrace \cap K_q$ lies entirely within the boundary of $K_q$ and intersects the positive $\theta$-axis, because of the particular form of $q$. Let $\eta\in S^{n-1}$ be the outward facing unit normal to $\lbrace q(x) = 0\rbrace\cap K_q$ (see Figure \ref{cones}). We then have
\begin{align*}
\lbrace q(x) = 0\rbrace\cap K_q = K_q\cap \lbrace h_{K_q}(\eta)\, \eta +\eta^\perp\rbrace
\end{align*}
and $\langle\theta,\eta\rangle >0$.
\begin{figure}[ht]
  \centering
  \includegraphics[scale=0.35]{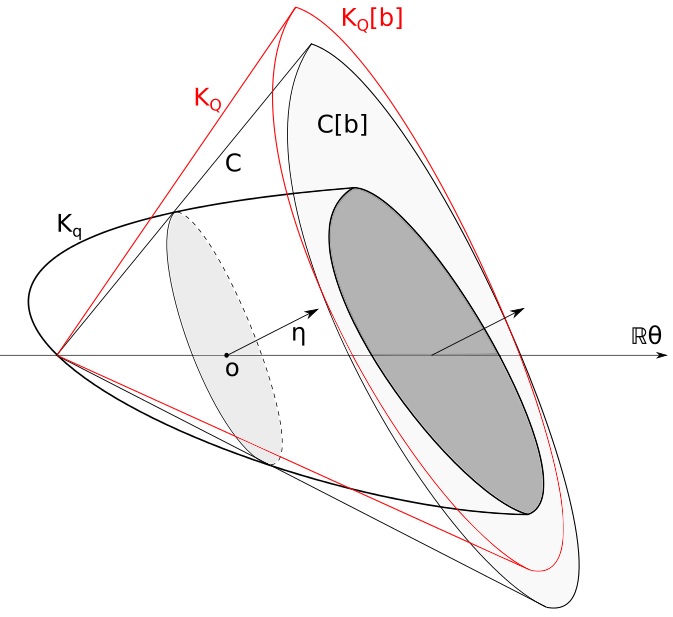}\\
  \caption{The construction of cones $C$ and $K_Q$.}\label{cones}
\end{figure}

Let $C$ be the $n$-dimensional cone with vertex $-\rho_{K_q}(-\theta)\, \theta\in K_q$, base lying in the hyperplane $\lbrace h_{K_q}(\eta)\, \eta +\eta^\perp\rbrace$, and for which $C\cap\eta^\perp = K_q\cap\eta^\perp$. Note that
\begin{align*}
-h_{K_q}(-\eta) \leq -h_C(-\eta) = -\rho_{K_q}(-\theta) \langle\theta,\eta\rangle
	< 0 < h_C(\eta) = h_{K_q}(\eta).
\end{align*}
The ``section volume" functions
\begin{align*}
A_{C,\eta} (t) = \mbox{vol}_{n-1} \big( C \cap \lbrace t\eta + \eta^\perp \rbrace \big), \quad
	A_{K_q,\eta} (t) = \mbox{vol}_{n-1} \big( K_q \cap \lbrace t\eta + \eta^\perp \rbrace \big) , \quad t\in\R,
\end{align*}
are $1/(n-1)$-concave by the Brunn-Minkowski inequality. In fact, an explicit calculation shows $A_{C,\eta}$ is $1/(n-1)$-affine. As we also have
\begin{align*}
A_{C,\eta}\Big(-h_C(-\eta)\Big) = 0 \leq A_{K_q,\eta}\Big(-h_C(-\eta)\Big) \quad \mbox{and} \quad
	A_{C,\eta}(0) = A_{K_q,\eta}(0) > 0 ,
\end{align*}
it is necessary that
\begin{align}\label{section inequality}
A_{C,\eta}(t) &\leq A_{K_q,\eta}(t) \quad \mbox{for all} \quad t \leq 0 , \nonumber \\
A_{C,\eta}(t) &\geq A_{K_q,\eta}(t) \quad \mbox{for all} \quad t \geq 0 .
\end{align}

For convenience put $a:= -h_C(-\eta)$, $b:= h_C(\eta)$, and define
\begin{align*}
C[t] := C\cap \lbrace t\eta + \eta^\perp \rbrace , \quad t\in\R .
\end{align*}
For each $t\in (a,b]$, $C[t]$ is an $(n-1)$-dimensional convex body whose centroid within the hyperplane $\lbrace t\eta + \eta^\perp \rbrace$ is given by
\begin{align*}
g\big(C[t]\big) = \left( \frac{ 1 }{ \mbox{vol}_{n-1}\big(C[t]\big) } \int_{C[t]} x\, dx \right)
	\in \lbrace t\eta + \eta^\perp \rbrace \subset \mathbb{R}^n ,
\end{align*}
in terms of the ambient coordinates of $\mathbb{R}^n$.

Now, define the cone
\begin{align*}
K_Q := \mbox{conv} \left( -\rho_{K_q}(-\theta)\, \theta , \
	C[b] - g\big(C[b]\big) + \frac{b}{ \langle\theta,\eta\rangle } \theta \right) .
\end{align*}
The cones $K_Q$ and $C$ have the same vertex, they have the same width in the direction $\eta$, and their sections $K_Q[t]$, $C[t]$ are translates lying in the same hyperplane $\lbrace t\eta + \eta^\perp\rbrace$. Therefore, the inequalities in (\ref{section inequality}) are valid for $A_{K_Q,\eta}$ in place of $A_{C,\eta}$. We also have
\begin{align}\label{C_t-centroid}
g(K_Q[t]) &= \left( \frac{ b-t }{ b-a } \right) \big(-\rho_{K_q}(-\theta)\theta \big)
	+ \left( 1 - \frac{ b-t }{ b-a } \right) g(K_Q[b]) \nonumber \\
	&= \left( \frac{ b-t }{ b-a } \right) \big(-\rho_{K_q}(-\theta)\theta \big) + \left( \frac{ t-a }{ b-a } \right) g(K_Q[b])
\end{align}
for all $t\in[a,b]$, because $K_Q[t]$ is a dilated and translated copy of $K_Q[b]$ with
\begin{align*}
K_Q[t] = \left( \frac{ b-t }{ b-a } \right) \big( -\rho_{K_q}(-\theta)\theta \big) + \left( \frac{ t-a }{ b-a } \right) K_Q[b] .
\end{align*}

Define the $\gamma$-affine function $Q: \R^n\rightarrow [0,\infty)$ by
\begin{align*}
Q(x) = \mathcal{X}_{K_Q} (x) \, q\left( \frac{ \langle x,\eta\rangle }{ \langle \theta,\eta\rangle } \theta \right) .
\end{align*}
The support of $Q$ is $K_Q$, $Q$ is constant on the sections $K_Q[t]$, and
\begin{align}\label{Q = q on Rtheta}
Q(t\theta) = q(t\theta) \qquad \mbox{for all} \qquad t\in\R .
\end{align}

\begin{lem}\label{Centroid on line}
There is a $0<\lambda_0<1$ so that
\begin{align*}
g(Q) = \lambda_0 \big(-\rho_{K_q}(-\theta)\theta \big) + (1-\lambda_0) g(K_Q[b])
\end{align*}
\end{lem}

\begin{proof}
First, note that
\begin{align*}
\int_{K_Q} x\, Q(x)\, dx = \int_a^b \int_{K_Q[t]} y\, Q(y)\, dy\, dt
	= \int_a^b q\left( \frac{ t\theta }{ \langle \theta,\eta\rangle } \right) A_{K_Q,\eta}(t) g(K_Q[t]) \, dt ,
\end{align*}
because $Q$ is constant on the sections $K_Q[t]$. Integration by parts then gives
\begin{align*}
\int_{K_Q} x\, Q(x)\, dx &= \left( \int_a^b q\left( \frac{ t\theta }{ \langle \theta,\eta\rangle } \right)
	A_{K_Q,\eta}(t) \, dt\right) g(K_Q[b]) \\
&\hspace{0.5cm}- \left( \int_a^b \int_a^t q\left( \frac{ t\theta }{ \langle \theta,\eta\rangle } \right)
	A_{K_Q,\eta}(s) \, ds\, dt \right) \left( \frac{ g(K_Q[b]) + \rho_{K_q}(-\theta)\theta }{ b-a } \right) ,
\end{align*}
where we use the representation of $g(K_Q[t])$ in (\ref{C_t-centroid}) to its derivative. Dividing both sides of the last equation by
\begin{align*}
\int_{K_Q} Q(x)\, dx = \int_a^b \int_{K_Q[t]} Q(y)\, dy\, dt
	= \int_a^b q\left( \frac{ t\theta }{ \langle \theta,\eta\rangle } \right) A_{K_Q,\eta}(t) \, dt
\end{align*}
and then rearranging the right-hand side shows
\begin{align*}
g(Q) = \frac{ \int_{K_Q} x\, Q(x)\, dx }{ \int_{K_Q} Q(x)\, dx }
	= \lambda_0 \big(-\rho_{K_q}(-\theta)\theta \big) + (1-\lambda_0) g(K_Q[b]) ,
\end{align*}
where
\begin{align*}
0< \lambda_0 :=
	\frac{ \int_a^b \int_a^t q \big( s\theta / \langle \theta,\eta\rangle \big) A_{K_Q,\eta}(s) \, ds\, dt }
	{ (b-a) \int_a^b q \big( t\theta / \langle \theta,\eta\rangle \big) A_{K_Q,\eta}(t) \, dt } < 1.
\end{align*}
\end{proof}

\begin{rem}\label{Centroid Remark}
It can be seen from the proof of Lemma \ref{Centroid on line} that any function which is
\begin{itemize}
\item integrable with a positive integral;
\item supported by a cone;
\item constant on hyperplane sections of the cone parallel to the base;
\end{itemize}
will have its centroid on the line connecting the vertex of the cone to the centroid of the base.
\end{rem}

\begin{lem}
The centroid $g(Q)$ lies on the $\theta$-axis and $\langle g(Q),\theta\rangle \geq 0$. Furthermore,
\begin{align}\label{Q ratio}
\frac{ \int_{\langle g(Q),\theta\rangle}^\infty Q(s\theta)\, ds}{\int_{-\infty}^\infty Q(s\theta)\, ds }
	\leq \frac{ \int_0^\infty Q(s\theta)\, ds}{\int_{-\infty}^\infty Q(s\theta)\, ds }
	= \frac{ \int_0^\infty q(s\theta)\, ds}{\int_{-\infty}^\infty q(s\theta)\, ds } ,
\end{align}
with equality if and only if $K_Q = K_q$ and $Q\equiv q$.
\end{lem}

\begin{proof}
Both the vertex of $K_Q$ and the centroid $g(K_Q[b])$ lie on the $\theta$-axis, so $g(Q) = t_0\theta$ for some $t_0\in\R$ by Lemma \ref{Centroid on line}. We have
\begin{align}\label{Q1}
0 = \int_{K_q} \langle x,\eta\rangle q(x)\, dx &= \int_{-\infty}^\infty \int_{ K_q[t] } \langle y,\eta\rangle q(y)\, dy\, dt \\
	&\leq \int_a^\infty \int_{ K_q[t] } \langle y,\eta\rangle q(y)\, dy\, dt
	= \int_a^\infty t\, q\left( \frac{ t\theta }{ \langle \theta,\eta\rangle } \right) A_{K_q,\eta}(t)\, dt , \nonumber
\end{align}
because $g(q) = o$, $-h_{K_q}(-\eta) \leq a := -h_C(-\eta) <0$, and $q$ has the constant value $q\big( t\theta / \langle \theta,\eta\rangle \big)$ on the section $K_q[t]$. Similarly,
\begin{align*}
\int_{K_Q} \langle x,\eta\rangle Q(x)\, dx
	= \int_a^\infty \int_{K_Q[t] } \langle y,\eta\rangle Q(y)\, dy\, dt
	= \int_a^\infty t\, q\left( \frac{ t\theta }{ \langle \theta,\eta\rangle } \right) A_{K_Q,\eta}(t)\, dt ,
\end{align*}
because $Q$ has the constant value $q\big( t\theta / \langle \theta,\eta\rangle \big)$ on the section $K_Q[t]$. Therefore,
\begin{align}\label{Q2}
&\int_{K_Q} \langle x,\eta\rangle Q(x)\, dx \nonumber \\
&\geq \int_a^\infty t\, q\left( \frac{ t\theta }{ \langle \theta,\eta\rangle } \right)
	\Big( A_{K_Q,\eta}(t) - A_{K_q,\eta}(t) \Big) \, dt \nonumber \\
&= \int_a^0 t\, q\left( \frac{ t\theta }{ \langle \theta,\eta\rangle } \right)
	\Big( A_{K_Q,\eta}(t) - A_{K_q,\eta}(t) \Big) \, dt \nonumber \\
&\hspace{3cm} + \int_0^\infty t\, q\left( \frac{ t\theta }{ \langle \theta,\eta\rangle } \right)
	\Big( A_{K_Q,\eta}(t) - A_{K_q,\eta}(t) \Big) \, dt \nonumber \\
&\geq 0 ,
\end{align}
using inequality (\ref{section inequality}) and the fact that $A_{K_Q,\eta}(t) = A_{C,\eta}(t)$. This shows
\begin{align*}
0\leq \langle g(Q),\eta\rangle = t_0 \langle \theta,\eta\rangle ,
\end{align*}
which then implies $t_0\geq 0$ because $\langle\theta,\eta\rangle\geq 0$. That is, $\langle g(Q),\theta\rangle \geq 0$. We get (\ref{Q ratio}) from $\langle g(Q),\theta\rangle \geq 0$ and (\ref{Q = q on Rtheta}).

Suppose there is equality in (\ref{Q ratio}). Necessarily $\langle g(Q),\theta\rangle = 0$, so there must also be equality in (\ref{Q2}) and (\ref{Q1}). Therefore, $-a = h_{K_Q}(-\eta) = h_{K_q}(-\eta)$ and $A_{K_Q,\eta} \equiv A_{K_q,\eta}$. This means $A_{K_q,\eta}$ is $\gamma$-affine and increasing from zero on $[a,b]$, which is only possible if $K_q$ is a cone with vertex $-\rho_{K_q}(-\theta)\theta$ and base $K_q[b] = K_q\cap \lbrace h_{K_q}(\eta)\eta +\eta^\perp\rbrace$. Recalling the construction of the cone $C$, we see that $C = K_q$. Because the centroid $g(q)$ and the vertex of $K_q$ are on the $\theta$-axis, Remark \ref{Centroid Remark} implies $g(K_q[b]) = g(C[b])$ is also on the $\theta$-axis. The choice of vertex and base for $K_Q$ now implies $K_Q = C = K_q$. Since for each $c>0$, $\{q(x)=c\}$ and $\{Q(x)=c\}$ lie in the same translate of $\eta^{\perp}$ and the supports of both functions coincide, we must have $Q \equiv q$.
\end{proof}

\begin{rem}
By applying the argument in this subsection to the function $q(x) = F(x+g(F))$ (where $F$ is defined in (\ref{F})), we can conclude 
$$
\frac{\int_0^{\infty} f(s \theta) ds}{\int_{-\infty}^{\infty} f(s \theta) ds} \geq \frac{\int_{\langle g(F), \theta\rangle}^{\infty} F(s \theta) ds}{\int_{-\infty}^{\infty} F(s \theta) ds}\geq \frac{\int_{\langle g(Q), \theta\rangle}^{\infty} Q(s \theta) ds}{\int_{-\infty}^{\infty} Q(s \theta) ds}.
$$
\end{rem}

\subsection{Equality case}

We will evaluate the last of the integrals in the previous remark. Fix any unit vector $\xi\in S^{n-1}$ with $\langle \theta,\xi\rangle > 0$. Consider any $n$-dimensional cone
\begin{align*}
K_T = \mbox{conv} \big( r_0\theta, \ r_1\theta + D \big) ,
\end{align*}
where $r_0,\, r_1\in\R$ with $r_0 < r_1$, and $D$ is an $(n-1)$-dimensional convex body in $\xi^\perp$ with $g(D)$ at the origin. Let $T:\R^n\rightarrow [0,\infty)$ be any $\gamma$-affine function having the form
\begin{align*}
T(x) = m \mathcal{X}_{K_T}(x) \big( - \langle x,\xi\rangle + r_1 \langle \theta,\xi\rangle \big)^\frac{1}{\gamma} ,
\end{align*}
where $m>0$ is a constant. We now determine the coordinates of $g(T)$. Compute
\begin{align*}
&\int_{K_T} \langle x,\xi\rangle T(x)\, dx \\
&= \int_{r_0\langle \theta,\xi\rangle}^{r_1\langle \theta,\xi\rangle}
	s\cdot m \cdot ( -s + r_1 \langle \theta,\xi\rangle)^\frac{1}{\gamma} A_{K_T,\xi}(s) \, ds \\
&= \int_{r_0\langle \theta,\xi\rangle}^{r_1\langle \theta,\xi\rangle} s\cdot m \cdot ( -s + r_1 \langle \theta,\xi\rangle)^\frac{1}{\gamma}
	\left( \frac{ s - r_0\langle\theta,\xi\rangle }{ (r_1-r_0) \langle\theta,\xi\rangle } \right)^{n-1}
	\mbox{vol}_{n-1}(D)\, ds \\
&= m \Big( (r_1-r_0)\langle\theta,\xi\rangle \Big)^{\frac{1}{\gamma}+2} \mbox{vol}_{n-1}(D)
	\int_0^1 t^n (1-t)^\frac{1}{\gamma}\, dt \\
&\hspace{2cm} + m r_0 \langle \theta,\xi\rangle \Big( (r_1-r_0)\langle\theta,\xi\rangle \Big)^{\frac{1}{\gamma}+1}
	\mbox{vol}_{n-1}(D) \int_0^1 t^{n-1} (1-t)^\frac{1}{\gamma}\, dt
\end{align*}
and
\begin{align*}
&\int_{K_T} T(x)\, dx \\
&= \int_{r_0\langle \theta,\xi\rangle}^{r_1\langle \theta,\xi\rangle} m
	( -s + r_1 \langle \theta,\xi\rangle)^\frac{1}{\gamma}
	\left( \frac{ s - r_0\langle\theta,\xi\rangle }{ (r_1-r_0) \langle\theta,\xi\rangle } \right)^{n-1}
	\mbox{vol}_{n-1}(D)\, ds \\
&= m \Big( (r_1-r_0)\langle\theta,\xi\rangle \Big)^{\frac{1}{\gamma}+1}
	\mbox{vol}_{n-1}(D) \int_0^1 t^{n-1} (1-t)^\frac{1}{\gamma}\, dt 
\end{align*}
using the change of variables $t = \frac{ s - r_0\langle\theta,\xi\rangle }{ (r_1-r_0) \langle\theta,\xi\rangle }$. Combining these calculations gives
\begin{align*}
\langle g(T),\xi\rangle = \frac{ \int_{K_T} \langle x,\xi\rangle T(x)\, dx }{ \int_{K_T} T(x)\, dx }
	&= (r_1-r_0)\langle\theta,\xi\rangle
	\left( \frac{ \int_0^1 t^n (1-t)^\frac{1}{\gamma}\, dt }{ \int_0^1 t^{n-1} (1-t)^\frac{1}{\gamma}\, dt } \right)
	+ r_0 \langle \theta,\xi\rangle \\
&= (r_1-r_0)\langle\theta,\xi\rangle \left( \frac{ n\gamma }{ (n+1)\gamma + 1 } \right) + r_0 \langle \theta,\xi\rangle,
\end{align*}
where we use the fact that for the Gamma function $\Gamma(z)$ one has 
\begin{align*}
\int_0^1 t^{u-1}(1-t)^{v-1}dt = \frac{\Gamma(u) \Gamma(v)}{\Gamma(u+v)} 
	\qquad \mbox{for all} \qquad u>0, \ v>0. 
\end{align*} 
Both the vertex of $K_T$ and the centroid of its base are on the $\theta$-axis, so
\begin{align*}
g(T) = \left[ (r_1-r_0) \left( \frac{ n\gamma }{ (n+1)\gamma + 1 } \right) + r_0 \right]\theta
	= \left[ \frac{ n\gamma r_1 + (\gamma +1) r_0 }{ (n+1)\gamma + 1 } \right] \theta \in \R\theta
\end{align*}
by Remark \ref{Centroid Remark}. Note that the centroid $g(T)$ will be at the origin if and only if
\begin{align*}
r_0 = - \left( \frac{ n\gamma }{ \gamma + 1 } \right) r_1 .
\end{align*}
Finally, calculate
\begin{align*}
\int_{\langle g(T),\theta\rangle}^\infty T(s\theta)\, ds
	&= \int_{\langle g(T),\theta\rangle}^{r_1}
	\Big( -s\langle \theta,\xi\rangle + r_1 \langle \theta,\xi\rangle\Big)^\frac{1}{\gamma} ds \\
&= \langle\theta,\xi\rangle^\frac{1}{\gamma} \left(\frac{\gamma}{\gamma +1}\right)
	(r_1 - r_0)^\frac{\gamma + 1}{\gamma} \left( \frac{\gamma +1}{ n\gamma + \gamma + 1 } \right)^\frac{\gamma +1}{\gamma}
\end{align*}
and
\begin{align*}
\int_{-\infty}^\infty T(s\theta)\, ds
	= \int_{r_0}^{r_1} \Big( -s\langle \theta,\xi\rangle + r_1 \langle \theta,\xi\rangle\Big)^\frac{1}{\gamma} ds
	= \langle\theta,\xi\rangle^\frac{1}{\gamma} \left(\frac{\gamma}{\gamma +1}\right) (r_1 - r_0 )^\frac{\gamma +1}{\gamma}
\end{align*}
to see that
\begin{align*}
\frac{ \int_{\langle g(T),\theta\rangle}^\infty T(s\theta)\, ds }{ \int_{-\infty}^\infty T(s\theta)\, ds }
	= \left( \frac{\gamma +1}{ n\gamma + \gamma + 1 } \right)^\frac{\gamma +1}{\gamma} .
\end{align*}

This concludes the proof of Theorem \ref{main}.

\section{ $k$-Dimensional Sections of $\gamma$-Concave Functions }

Recall $\theta^+ := \lbrace x\in\R^n :\, \langle x,\theta\rangle\geq 0\rbrace$ for $\theta\in S^{n-1}$. We have the following generalization:

\begin{cor}\label{general prob}
Fix a $k$-dimensional subspace $E$ of $\mathbb{R}^n$, $\theta\in E\cap S^{n-1}$, and $\gamma\in (0,\infty)$. Let $f:\mathbb{R}^n\rightarrow [0,\infty)$ be a $\gamma$-concave function with $0<\int_{\R^n} f(x)\, dx<\infty$ and centroid at the origin. Then 
\begin{align*}
\frac{ \int_{E \cap \theta^+} f(x) dx }{ \int_E f(x) dx } 
	\geq \left(\frac{k \gamma+1}{(n+1) \gamma+1}\right)^{\frac{k \gamma+1}{\gamma}} .
\end{align*}
There is equality when 
\begin{itemize}
\item $f(x) = \mathcal{X}_{K_f}(x) \Big( - \langle x, \theta\rangle + 1 \Big)^\frac{1}{\gamma}$; 
\item $K_f = \mbox{conv} \left( - (n-k+1) \left( \frac{ \gamma }{\gamma +1}\right) \theta + \delta B_2^{k-1} , \ \theta + B_2^{n-k} \right)$ where $B_2^{n-k}$ is the centred Euclidean ball of unit radius in $E^\perp$, $B_2^{k-1}$ is the centred Euclidean ball of unit radius in $\widetilde{E}^\perp = E\cap\theta^\perp$, and 
\begin{align*}
\delta = \left( (n-k+1) \left( \frac{ \gamma }{\gamma +1}\right) + 1 \right) 
	\Big[ \mbox{vol}_{k-1} \big( B_2^{k-1} \big) \Big]^\frac{-1}{k-1} . 
\end{align*}
\end{itemize}
\end{cor}

\begin{proof}
Put $\widetilde{E} = \mbox{span} \{ E^{\perp}, \theta\}$, and define the function $F: \widetilde{E} \to [0,\infty)$ by
\begin{align*}
F(y) := \int_{\widetilde{E}^{\perp}} f(z + y)\, dz.
\end{align*}
We claim $F$ is a $\widetilde{\gamma}:= \frac{\gamma}{(k-1) \gamma + 1}$-concave function on the $d:= (n-k+1)$-dimensional space $\widetilde{E}$. Fix any $y_1,\, y_2\in\widetilde{E}$ with $F(y_1)\cdot F(y_2)\neq 0$ and $0<\lambda<1$. The $\gamma$-concavity of $f$ allows us to apply the Borell-Brascamp-Lieb inequality to the functions 
\begin{align*}
z \mapsto f\big( z + \lambda y_1 + (1-\lambda) y_2 \big) , 
	\quad z \mapsto f( z + y_1 ) , 
	\quad z \mapsto f( z + y_2 ) 
\end{align*}
on $\widetilde{E}^\perp\in G(n,k-1)$ to get 
\begin{align*}
&F \big(\lambda y_1 + (1-\lambda) y_2 \big) = \int_{\widetilde{E}^{\perp}} f\big(z +\lambda y_1 + (1-\lambda) y_2 \big)\, dz \\
&\geq \left(\lambda \left(\int_{\widetilde{E}^{\perp}} f(z+y_1)\, dz\right)^{\frac{\gamma}{(k-1) \gamma + 1}} 
	+ (1-\lambda) \left(\int_{\widetilde{E}^{\perp}} f(z+y_2)\, dz\right)^{\frac{\gamma}{(k-1) \gamma + 1}} 
	\right)^{\frac{(k-1)\gamma+1}{\gamma}} \\
&= \left(\lambda F^{\frac{\gamma}{(k-1) \gamma + 1}}(y_1) 
	+ (1 - \lambda) F^{\frac{\gamma}{(k-1) \gamma + 1}}(y_2) \right)^{\frac{(k-1)\gamma+1}{\gamma}}.
\end{align*}
Observe that $g(F) = o$. Indeed,
\begin{align*}
\int_{\widetilde{E}} y\, F(y)\, dy &= \int_{\widetilde{E}} y\, \int_{\widetilde{E}^{\perp}} f(z+y)\,dz\, dy \\
&= \int_{\R^n} \big( x|\widetilde{E} \big) f(x) \, dx 
	= \left( \int_{\R^n} x \, f(x) \, dx \right) \bigg| \widetilde{E} = o .
\end{align*}
By Theorem \ref{main}, we have 
\begin{align*}
\frac{ \int_{E \cap \theta^+} f(x)\, dx  }{\int_E f(x)\, dx } 
	= \frac{\int_0^{\infty} F(s \theta)\, ds}{\int_{-\infty}^{\infty} F(s \theta)\, ds} 
	\geq \left(\frac{\widetilde{\gamma}+1}{(d+1) \widetilde{\gamma} +1}\right)^{(\widetilde{\gamma}+1)/\widetilde{\gamma}} 
	= \left(\frac{k \gamma+1}{(n+1) \gamma+1}\right)^{\frac{k \gamma+1}{\gamma}} .
\end{align*}

Assume $f(x) = \mathcal{X}_{K_f}(x) \Big( - \langle x, \theta\rangle + 1 \Big)^\frac{1}{\gamma}$ and 
\begin{align*}
K_f = \mbox{conv} \left( - (n-k+1) \left( \frac{ \gamma }{\gamma +1}\right) \theta + \delta B_2^{k-1} , 
	\ \theta + B_2^{n-k} \right) . 
\end{align*}
Let $y$ be any point lying in the $(n-k+1)$-dimensional cone 
\begin{align*}
K_f|\widetilde{E} = \mbox{conv} \left( - (n-k+1) \left( \frac{ \gamma }{\gamma +1}\right) \theta ,
	\ \theta + B_2^{n-k} \right) 
\end{align*}
in $\widetilde{E}$. There is a point $v_1$ in the base $\theta + B_2^{n-k}$ of $K_f|\widetilde{E}$ so that $y$ lies on the line segment connecting $v_1$ to the vertex $v_0 := - (n-k+1) \left( \frac{\gamma}{\gamma +1}\right) \theta$ of $K_f|\widetilde{E}$. Then  
\begin{align*}
K_f \cap \mbox{aff} \left( v_0 + \widetilde{E}^\perp, \ v_1 \right) = \mbox{conv} \left( v_0 + \delta B_2^{k-1} , \ v_1 \right),
\end{align*}
and so 
\begin{align*}
\mbox{vol}_{k-1} \left( K_f \cap \left\lbrace y + \widetilde{E}^\perp \right\rbrace \right) 
	&= \mbox{vol}_{k-1} \left( \mbox{conv} \left( v_0 + \delta B_2^{k-1} , \ v_1 \right) 
	\cap \left\lbrace y + \widetilde{E}^\perp \right\rbrace \right) \\
&= \left( \frac{ \langle v_1 - y,\theta\rangle }{ \langle v_1 - v_0,\theta\rangle } \right)^{k-1} 
	\mbox{vol}_{k-1} \big( \delta B_2^{k-1} \big) \\ 
&= \Big( -\langle y, \theta \rangle + 1 \Big)^{k-1} .
\end{align*} 
The function $F: \widetilde{E}\rightarrow [0,\infty)$, defined by
\begin{align*}
F(y) := \int_{\widetilde{E}^\perp} f(z+y) \, dz 
	&= \int_{\widetilde{E}^\perp} \mathcal{X}_{K_f}(z+y) \Big( - \langle z + y, \theta\rangle + 1 \Big)^\frac{1}{\gamma} dz \\
&= \Big( - \langle y, \theta\rangle + 1 \Big)^\frac{1}{\gamma} \int_{\widetilde{E}^\perp} \mathcal{X}_{K_f}(z+y) \, dz \\
&= \Big( - \langle y, \theta\rangle + 1 \Big)^\frac{1}{\gamma} 
	\mbox{vol}_{k-1} \left( K_f \cap \left\lbrace y + \widetilde{E}^\perp \right\rbrace \right) \\
&= \Big( - \langle y, \theta\rangle + 1 \Big)^\frac{1}{\gamma} 
	\mathcal{X}_{K_f|\widetilde{E}}(y) \Big( - \langle y, \theta\rangle + 1 \Big)^{k-1} \\
&= \mathcal{X}_{K_f|\widetilde{E}}(y) \Big( - \langle y, \theta\rangle + 1 \Big)^\frac{ (k-1)\gamma + 1 }{\gamma} ,
\end{align*}
is then $\widetilde{\gamma}$-affine with support 
\begin{align*}
K_f|\widetilde{E} = \mbox{conv} \left( - (n-k+1) \left( \frac{ \gamma }{\gamma +1}\right) \theta , \ \theta + B_2^{n-k} \right) .
\end{align*}
The centroid of $f$ must lie in $\widetilde{E}$, because $f$ is symmetric with respect to $\widetilde{E}$. Also notice that $F$ satisfies the equality conditions of Theorem \ref{main} in dimension $n-k+1$ for $\theta = \xi$ and $r=1$. Therefore, the centroids of $F$ and $f$ are at the origin, and 
\begin{align*}
\frac{ \int_{E \cap \theta^+} f(x)\, dx  }{\int_E f(x)\, dx }
	= \frac{ \int_0^\infty F(s\theta)\, ds }{ \int_{-\infty}^\infty F(s\theta)\, ds } 
	= \left(\frac{k \gamma+1}{(n+1) \gamma+1}\right)^{\frac{k \gamma+1}{\gamma}}. 
\end{align*}
\end{proof}

\section{ Sections of Convex Bodies }

We have the following corollary to Theorem \ref{main}:

\begin{cor}\label{Grunbaum Ineq Sec Thm}
Fix a $k$-dimensional subspace $E$ of $\R^n$, and $\theta\in E\cap S^{n-1}$. Let $\widetilde{E}$ be the $(n-k+1)$-dimensional subspace spanned by $\theta$ and $E^\perp$. Let $K$ be a convex body in $\R^n$ with $g(K)\in \widetilde{E}^\perp = E\cap\theta^\perp$. Then
\begin{align*}
\frac{ \mbox{vol}_k (K\cap E\cap\theta^+) }{ \mbox{vol}_k (K\cap E) } \geq \left( \frac{k}{n+1} \right)^k .
\end{align*}
There is equality if and only if
\begin{align*}
K = \mbox{conv}\left( -\left( \frac{ n-k+1 }{ k } \right) z + D_0,\, z + D_1\right),
\end{align*}
where
\begin{itemize}
\item $z\in E$ with $\langle z,\theta\rangle > 0$;
\item $D_0$ is a $(k-1)$-dimensional convex body in $\widetilde{E}^\perp$;
\item $D_1$ is an $(n-k)$-dimensional convex body in an $(n-k)$-dimensional subspace $F\subset\R^n$ for which $\R^n = \mbox{span}(E,F)$, and $g(D_1)$ is at the origin (see Figure \ref{equal1}).
\end{itemize}
\end{cor}

\begin{figure}[ht]
  \centering
  \includegraphics[scale=0.35]{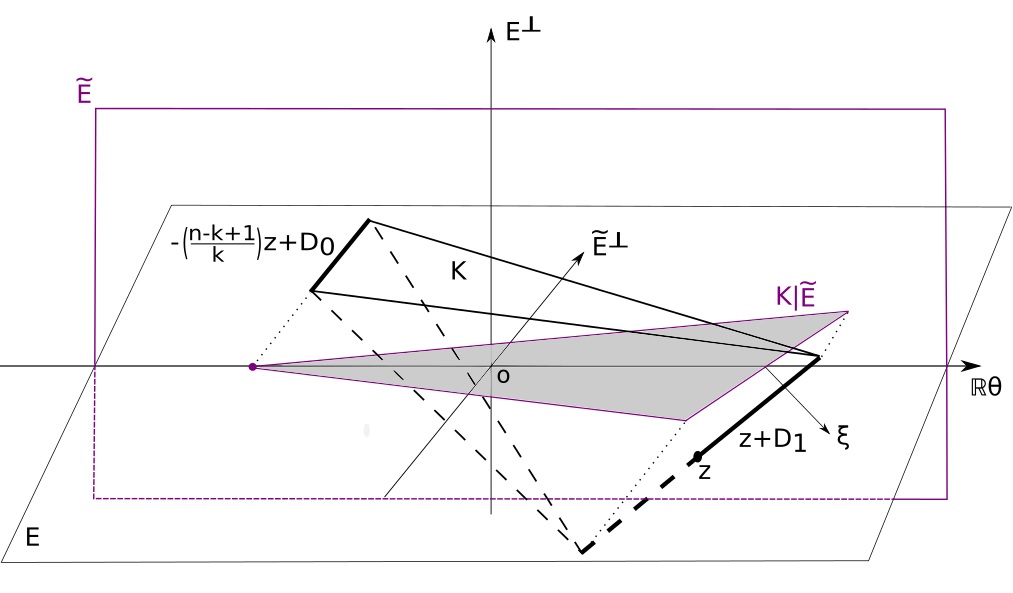}\\
  \caption{The equality case.}\label{equal1}
\end{figure}

\begin{proof}
Define the section function
\begin{align*}
A_{K,\widetilde{E}} (y) := \mbox{vol}_{k-1} \left( K\cap \lbrace y + \widetilde{E}^\perp \rbrace \right), \quad y \in \widetilde{E}.
\end{align*}
It follows from the Brunn-Minkowski inequality that $A_{K,\widetilde{E}}:\widetilde{E}\rightarrow [0,\infty)$ is a $\gamma:= 1/(k-1)$-concave function on the $d := (n-k+1)$-dimensional space $\widetilde{E}$. The centroid of $A_{K,\widetilde{E}}$ is at the origin; indeed,
\begin{align*}
g\big( A_{K,\widetilde{E}} \big)
	= \frac{ \int_{K|\widetilde{E}} y\, A_{K,\widetilde{E}}(y) \, dy }{ \int_{K|\widetilde{E}} A_{K,\widetilde{E}} \, dy }
	= \frac{ \int_K ( x|\widetilde{E} ) \, dx }{ \mbox{vol}_n (K) }
	= g(K) | \widetilde{E} = o .
\end{align*}
From Theorem \ref{main},
\begin{align*}
\frac{ \mbox{vol}_k (K\cap E\cap\theta^+) }{ \mbox{vol}_k (K\cap E) }
	= \frac{ \int_0^\infty A_{K,\widetilde{E}}(s\theta)\, ds }{ \int_{-\infty}^\infty A_{K,\widetilde{E}}(s\theta)\, ds }
	\geq \left( \frac{ \gamma + 1 }{ \gamma d +\gamma + 1 } \right)^\frac{ \gamma +1 }{\gamma} = \left( \frac{k}{n+1} \right)^k
\end{align*}
with equality if and only if
\begin{itemize}
\item $A_{K,\widetilde{E}}(y) = m \mathcal{X}_{K|\widetilde{E}}(y) \Big( - \langle y,\xi\rangle + r \langle \theta,\xi\rangle \Big)^\frac{1}{\gamma}$ for some constants $m,r >0$ and a unit vector $\xi\in \widetilde{E}\cap S^{n-1}$ such that $\langle\theta,\xi\rangle >0$;
\item $K|\widetilde{E} = \mbox{conv} \left( - \left( \frac{d\gamma}{\gamma +1}\right) r\theta, \ r \theta + D \right)$ for some $(d-1)$-dimensional convex body $D\subset\widetilde{E}\cap\xi^\perp$ whose centroid is at the origin.
\end{itemize}
These equality conditions are equivalent to the ones given in the corollary statement, where $m$ is the $(k-1)$-dimensional volume of $D_0$, $r = \langle z,\theta\rangle$, $\widetilde{E}\cap \xi^\perp = F|\widetilde{E}$, and $D = D_1|\widetilde{E}$. 
\end{proof}

\begin{rem}
Observe that the inequality in Corollary \ref{Grunbaum Ineq Sec Thm} is the limiting case of the inequality in Corollary \ref{general prob} as $\gamma$ goes to infinity. This corresponds to the fact that $\infty$-concave functions, defined by taking the limit in (\ref{gamma def}), are the indicator functions of convex sets. 
\end{rem}

\begin{rem}
We are able to recover Gr\"unbaum's inequality for projections from Gr\"unbaum's inequality for sections. Consider any convex body $K\subset\mathbb{R}^n$ with its centroid at the origin. Let $\widetilde{K}$ be the Steiner symmetrization of $K$ with respect to the $k$-dimensional subspace $E\subset\mathbb{R}^n$. Specifically, 
\begin{align*}
\widetilde{K} = \bigcup_{y\in K|E} \left\{ y + \left( \frac{ \mbox{vol}_{n-k} \left( K \cap \lbrace y + E^\perp \rbrace \right) }
	{ \mbox{vol}_{n-k} \left( B_2^{n-k} \right) } \right)^\frac{1}{n-k} B_2^{n-k} \right\} ,
\end{align*}
where $B_2^{n-k}$ is the centred Euclidean ball of unit radius in $E^\perp$. Now, $\widetilde{K}$ is a convex body with its centroid at the origin, and 
\begin{align*}
\widetilde{K}\cap E \cap \theta^+ = \big( K|E \big) \cap \theta^+ \quad \mbox{for all} \quad \theta\in E\cap S^{n-1} .
\end{align*}
\end{rem}

\section*{Acknowledgements}

\noindent This material is based upon work supported by the National Science Foundation under Grant No. DMS-1440140 while the second and third named authors were in residence at the Mathematical Sciences Research Institute in Berkeley, California, during the Fall 2017 semester. The authors were also partially supported by a grant from NSERC. \\

\noindent The authors thank Dmitry Ryabogin and Vladyslav Yaskin for their guidance. They also thank Beatrice-Helen Vritsiou for reading the manuscript and providing helpful comments.


\begin{thebibliography}{00}


\bibitem{BF} {\sc T.~Bonnesen and W.~Fenchel}, {\em Theory of convex bodies}, BCS Associates, Moscow, ID, 1987.


\bibitem{BGVV} {\sc S.~Brazitikos, A.~Giannopoulos, P.~Valettas, and B.H.~Vritsiou}, {\em Geometry of Isotropic Convex Bodies}, American Mathematical Society, Providence, RI, 2014.


\bibitem{FMY} {\sc M.~Fradelizi, M.~ Meyer, and V.~Yaskin}, {\em On the volume of sections of a convex body by cones}, Proc. Amer. Math. Soc. {\bf 145} (2017), 3153--3164.


\bibitem{Ga1} {\sc R.~Gardner}, {\em The Brunn-Minkowski inequality}, Bull. Amer. Math. Soc. (N.S.) {\bf 39} (2002), 355--405.


%\bibitem{Ga2} {\sc R.~Gardner}, {\em Geometric Tomography}, second edition, Cambridge University Press, New York, 2006.


\bibitem{G1} {\sc B.~Gr\"unbaum}, {\em Partitions of mass-distributions and of convex bodies by hyperplanes}, Pacific J. Math. {\bf 10} (1960), 1257--1261.


%\bibitem{G2} {\sc B.~Gr\"unbaum}, {\em Measures of symmetry for convex sets}, In: Proc. Sympos. Pure Math., Vol. VII, pp. 233--270, Amer. Math. Soc., Providence, R.I., 1963.


%\bibitem{LV} {\sc L.~Lov\'asz and S.~Vempala}, {\em The geometry of logconcave functions and sampling algorithms}, Random Structures and Algorithms {\bf 30} (2007), 307--358.


\bibitem{MNRY} {\sc M.~Meyer, F.~Nazarov, D.~Ryabogin, and V.~Yaskin}, {\em Generalized Gr\"unbaum inequality}, available at arXiv:1706.02373 [math.MG].


\bibitem{M} {\sc B.S.~Mityagin}, {\em Two inequalities for volumes of convex bodies (Russian)}, Mat. Zametki {\bf 5} (1969), 99--106.


\bibitem{SZ} {\sc M.~Stephen and N.~Zhang}, {\em Gr\"unbaum's inequality for projections}, J. Funct. Anal. {\bf 272} (2017), 2628--2640.


\end{thebibliography}
\end{document}